 \documentclass[final,1p,times]{elsarticle}

\usepackage{amssymb}
\usepackage{amsthm}
\usepackage[english]{babel} \usepackage{latexsym,amsfonts,amsmath,mathrsfs}
\newtheorem {theorem}{Theorem}[section]         \newtheorem {lemma}[theorem]{Lemma}     \newtheorem {definition}[theorem]{Definition}
     \newtheorem {remark}[theorem]{Remark}   
       \newtheorem {proposition}[theorem]{Proposition} 

\newcommand{\C}{\mathbb C}
\newcommand{\R}{\mathbb R}
\newcommand{\Z}{\mathbb Z} 	
\newcommand{\N}{\mathbb N}
\newcommand{\B}{\mathbb{B}}
\newcommand{\Bn}{\mathbb{B}^n}
\newcommand{\Br}{\mathbb{B}_R^2}
\newcommand{\Bnr}{\mathbb{B}^n_R}

\newcommand{\norm}[1]{\left\Vert#1\right\Vert}
\newcommand{\scal}[1]{\left<#1\right>}
\newcommand{\meslambda}{d\lambda}
\newcommand{\mes}{d\mu_\alpha}

\newcommand{\muar}{d\mu_{\alpha,R}}

\newcommand{\HolB}{\mathcal{H}ol(\Bn)}

\newcommand{\HolC}{\mathcal{H}ol(\C^n)}
\newcommand{\ADmn}{\mathcal{A}^{2,\alpha}_{m}(\B^n)}
\newcommand{\ADmnr}{\mathcal{A}^{2,\alpha}_{m}(\B^n_R)}

\newcommand{\ADm}{\mathcal{A}^{2,\alpha}_{m}(\B^2)}

\newcommand{\ADmr}{\mathcal{A}^{2,\alpha}_{R,m}(\Br)}

\newcommand{\BDmn}{\mathcal{F}^{2,\nu}_{m}(\C^n)}

\newcommand{\BDm}{\mathcal{F}^{2,\nu}_{m}(\C^2)}

\begin{document}

\begin{frontmatter}



\title{Weighted Bergman-Dirichlet and Bargmann-Dirichlet spaces of order $m$ in high dimensions}
\author{A. El Fardi}    \ead{aelfardi@gmail.com}
\author{A. Ghanmi}      \ead{ag@fsr.ac.ma}
\author{A. Intissar}    \ead{intissar@fsr.ac.ma}
\author{M. Ziyat}       \ead{ziatemohammed@gmail.com}
\address{E.D.P. and Spectral Geometry,
          Laboratory of Analysis and Applications-URAC/03,
          Department of Mathematics, P.O. Box 1014,  Faculty of Sciences,
          Mohammed V University, Rabat, Morocco}

\begin{abstract}
We introduce and study a generalization of the classical weighted Bergman and Dirichlet spaces on the unit ball in high dimension, the Bergman-Dirichlet spaces. Their counterparts on the whole $n$-complex space $\C^n$, the Bargmann-Dirichlet spaces, are also introduced and studied. Mainly, we give a complete description of the considered spaces, including orthonormal basis and the explicit formulas for their reproducing kernel functions. Moreover, we investigate their asymptotic  behavior when the curvature goes to $0$.
\end{abstract}


\end{frontmatter}



\section{Introduction and statement of main results}

The Segal-Bargmann space, on the $n$-complex space $\C^n$ endowed with its standard inner product $\scal{z,w}$, and the so-called weighted Bergman and Dirichlet spaces, on the open unit ball $\Bn=\{ z\in \C^n; \, |z|:= \sqrt{\scal{z,z}}< 1\}$, are basic examples of functional spaces in the theory of analytic functions.
 Such spaces play important roles in function theory and operator theory, as well as in modern analysis, probability and statistical analysis.
For a nice introduction and surveys of these spaces in the context of function and operator theories, see for example
\cite{Krantz19922001,RochbergWu1993,Wu1998,Zhu2005,ArcozziRochbergSawyerWick2011,Zhu2012,ElFallahKellayMashreghiRansford2014} and the references therein.

Recently, two new classes of analytic function spaces of Sobolev type,
 labeled by a nonnegative integer $m$, have been introduced and studied in \cite{EGIMS15}.
 The first one is the Bergman-Dirichlet space generalizing the weighted Bergman and
 Dirichlet spaces on the disk $D(0,R)$ in the complex plane $\C$. The second is the Bargmann-Dirichlet space generalizing the Segal-Bargmann space
 on the complex plane $\C= D(0,+\infty)$.
They are reproducing kernel Hilbert spaces. Their reproducing kernel functions have been calculated explicitly and expressed in terms
of the hypergeomtric functions.


Our purpose in the present paper is to introduce the spaces $\ADmn$, the analogue of the considered Bergman-Dirichlet spaces in high dimension.
What we do in the construction of $\ADmn$ works mutatis mutandis to introduce
their counterparts on the whole $n$-complex space $\C^n$, the Bargmann-Dirichlet spaces $\BDmn$ of order $m$. We investigate their spectral properties and generalize the results obtained in \cite{EGIMS15} to high dimensions $n\geq 1$.
A part of some special techniques introduced in the calculation, the approach used here to prove our main results is quite similar to the one-dimensional setting. The motivations for studying such a generalization are various and meaningful.

The paper is organized as follows. In Section 2, we introduce the weighted Bergman-Dirichlet and Bargmann-Dirichlet spaces and state our main results (Theorems \ref{Mthm1} and \ref{Mthm2}) giving the explicit expression of the corresponding  reproducing kernel functions.
Section 3 is devoted to the concrete description of the Bergman-Direchlet spaces and to the proof of Theorem \ref{Mthm1}.
In Section 4 we are concerned with the Bargmann-Dirichlet spaces and the proof of
Theorem \ref{Mthm2}.
We conclude the paper by studying the asymptotic behavior of the $L^2$-eigenprojector kernel of $\ADmnr$ when $R$ tends to infinity, and show that it gives rise to the $L^2$-eigenprojector kernel of $\BDmn$.

\section{Statement of main results}
For given $\alpha > -1$, we consider the measure
$ \mes(z) := \left(1-\left|z\right|^{2}\right)^{\alpha}\meslambda(z) ,$
where $\meslambda$ stands for the Lebesgue measure.
Then, the weighted Bergman space $\mathcal{A}^{2,\alpha}(\Bn)$ (\cite{Krantz19922001,Stroethoff1998,HuZhang2001,Zhu2005})
can be defined as the functional space of all holomorphic functions $f$ on $\Bn$, $f\in \HolB$, that are $\mes$-square integrable,
$ f\in L^{2,\alpha} (\Bn):=L^{2}\left(\Bn; \mes\right)$. That is
\begin{align}\label{wBs}
\mathcal{A}^{2,\alpha}(\Bn) := L^{2,\alpha}(\Bn)  \cap \HolB.
\end{align}
It is supplied with the norm $\norm{\cdot}_{\alpha}:=\norm{\cdot}_{L^{2,\alpha}(\Bn)}$ associated to the inner product defined as
\begin{align}\label{normBergman}
\scal{f,g}_{\alpha} := \int_{\Bn} f(z)\overline{g(z)} (1-|z|^2)^\alpha \meslambda(z) .
\end{align}
The weighted Dirichlet space $\mathcal{D}^{2,\alpha}(\Bn)$  is the analytic function space on $\Bn$ defined by \cite{Stevic2010,Zhu2005}
\begin{align}\label{wDs}
\mathcal{D}^{2,\alpha}(\Bn) := \left\lbrace  f(z)=\sum\limits_{p\in (\Z^+)^n} a_p z^p; \,\sum_{p\in (\Z^+)^n}^{+\infty}  |p| \frac{p!}{|p|!} |a_p|^2 < + \infty \right\rbrace
\end{align}
Here $|p|=p_1+ \cdots + p_n$ and $p!=p_1! \cdots  p_n!$ for given multi-index $p=(p_1, \cdots , p_n) \in (\Z^+)^n$ and $z^p =z_1^{p_1}  \cdots  z_n^{p_n}$ for given $z = (z_1, \cdots , z_n) \in \C^n$.

Now, for every fixed nonnegative integer $m$, every complex valued
holomorphic function $f$ with the Taylor expansion
$f(z)=\sum\limits_{p\in(\Z^+)^n}a_pz^p$
can be written as
\begin{align}\label{split}
f(z)=f_{1,m}(z) + f_{2,m}(z) ,
\end{align}
where $f_{1,m}$ and $f_{2,m}$ stand for
$f_{1,m}(z)=\sum_{|p|<m}a_pz^p $ and $ f_{2,m}(z)=\sum_{|p|\geq m}a_pz^p.$
We define
\begin{align}\label{norm}
\norm{f}_{\alpha,m}^2=\norm{f_{1,m}}_\alpha^2+m! \sum_{|p|=m}\frac{1}{p!}\norm{D^p f_{2,m}}_\alpha^2
\end{align}
to be the norm on $L^{2,\alpha}(\B^n)$ associated to the inner product
\begin{align}\label{innerp}
\scal{f,g}_{\alpha,m}= \scal{f_{1,m},g_{1,m}}_\alpha +  m!\sum_{|p|=m}\frac{1}{p!}\scal{D^p f_{2,m},D^p g_{2,m}}_\alpha.
\end{align}
Here $D^p$ is the partial differential operator of total degree $|p|$, $p=(p_1,\cdots,p_n)\in(\Z^+)^n$, defined by
\begin{align}\label{Dop}
D^p=\frac{\partial^{|p|}}{\partial z_1^{p_1}\cdots \partial z_n^{p_n}}.
\end{align}

\begin{definition}
	The functional space
$ 
	\ADmn =\left\{f\in\mathcal{H}ol(\B^n)\;\;/\;\; \norm{f}_{\alpha,m}<+\infty\right\}
$ 
	endowed with the norm \eqref{norm} will be called the generalized Bergman-Dirichlet space of order $m$ on $\B^n$.
\end{definition}

\begin{remark}\label{Rem-1} For the special cases $m=0$ and $m=1$, we recover the weighted Bergman space \eqref{wBs} and the classical Dirichlet space \eqref{wDs}, respectively.
\end{remark}

Our central result for these spaces can be stated as follows
\begin{theorem}\label{Mthm1} Keep notations as above. Then, the space $\ADmn$ is a reproducing kernel Hilbert space.
		Its reproducing kernel function is given in terms of the ${_3F_2}$-hypergeometric function by the following closed form
		\begin{align}
		K_{\ADmn}(z,w) =\frac{\Gamma(\alpha+n+1)}{\pi^n\Gamma(\alpha+1)}
		 & \sum\limits_{k <m}  (\alpha+n+1)_k\frac{\scal{z,w}^k}{k!}
		\label{RK-BDC}
		\\ &  + \frac{\Gamma(\alpha+n+1)}{\pi^n\Gamma(\alpha+1)} \frac{\scal{z,w}^{m}}{(m!)^{2}} \,
		{_3F_2}\left(\begin{array}{c} 1, 1, \alpha+n+1\\ m+1,m+1 \end{array}\bigg |  \scal{z,w} \right) .\nonumber
		\end{align}
\end{theorem}

\smallskip
On the whole $n$-dimensional euclidean complex space $\C^n$ and for given $\nu>0$, we denote by $L^{2,\nu}(\C^n):=L^2(\C^n; e^{-\nu|z|^2} \meslambda) $ the space of $e^{\nu|z|^2}\meslambda$-square integrable functions on $\C^n$. The Segal-Bargmann Hilbert space $\mathcal{F}^{2,\nu}(\C^n)$  is then defined to be  the space of all holomorphic functions belonging to $L^{2,\nu}(\C^n)$. That is
$ 
\mathcal{F}^{2,\nu}(\C^n)=  L^{2,\nu}(\C^n) \cap \HolC,
$ 
supplied with the norm
\begin{align}\label{BFnorm}
\norm{f}_{\nu}^2:=\int_{\C^n}|f(z)|^2e^{-\nu|z|^2} \meslambda(z) .
\end{align}
As a generalization of this space, we perform the space of all holomorphic functions on
$\C^n$ such that
\begin{equation}\label{norm2}
||f||_{\nu,m}^2=||f_{1,m}||_\nu^2+\sum_{|l|=m}\frac{m!}{l!}||D^l f(z)||_\nu^2  \ \ \ < +\infty,
\end{equation}
where we have split $f$ as in \eqref{split}, to wit $f = f_{1,m} + f_{2,m} $.
That is

\begin{definition}
	The functional space
$ 
	\BDmn=\left\{f\in\mathcal{H}ol(\C^n)\;\;/\;\;||f||_{\nu,m}<+\infty\right\}
$ 
	endowed with the norm \eqref{norm2} will be called the generalized Bargmann-Dirichlet space of order $m$ on $ \C^n $.
\end{definition}

$$\BDmn=\left\{f\in\mathcal{H}ol(\C^n)\;\;/\;\;||f||_{\nu,m}<\infty\right\}.$$
The analogue of Theorem \ref{Mthm1} for $\BDmn$ is the following

\begin{theorem}\label{Mthm2}
The space $\BDmn$ is a reproducing kernel Hilbert space.
   Its reproducing kernel function is given in terms of the ${_2F_2}$-hypergeometric function by
\begin{align}\label{RK-BDC}
K_{\BDmn}(z,w)=\left(\frac{\nu}{\pi}\right)^n
    \left(  \sum\limits_{k<m} \frac{(\nu \scal{z,w})^{k}}{k!}+\frac{\scal{z,w}^m}{(m!)^2} \,
   {_2F_2}\left(\begin{array}{c} 1, 1\\ m+1,m+1 \end{array}\bigg |  \nu  \scal{z,w} \right)\right) .
   \end{align}
\end{theorem}

%
%

What we have done for the unit ball can be extended in an appropriate way to any $0$-centered ball $\Bnr$ of radius $R$.
Doing so, one shows that the Bargmann-Dirichlet spaces $\BDmn$ can be seen as the limit
of the Bergman-Dirichlet spaces $\ADmnr$, with $\alpha = \nu R^{2}$, as $R$ goes to infinity, in the sense that we have

\begin{theorem}\label{Mthm3}
For every fixed nonegative integer $m$ and real number $\nu>0$, the reproducing kernel
 $K_{\ADmnr}$ of the weighted Bergman-Dirichlet space $\ADmnr$, with $\alpha=\nu R^2$, converges
 pointwisely and uniformly on compact sets of $\C^n\times \C^n$ to the reproducing kernel function $K_{\BDmn}$ of
weighted Bargmann-Dirichlet space $\BDmn$.
\end{theorem}

This is motivated by the fact that the flat Hermitian geometry on $\C^n$
can be approximated by the complex hyperbolic geometry of the balls $\Bnr$ of radius $R>0$
associated to an appropriate scaled Bergman K\"ahler metric \cite{GhIn2005JMP}. Indeed, we have to deal with the scaled measure
$$ \muar(z) := \left(1-\left|\frac{z}{R^2}\right|^{2}\right)^{\nu R^2}\meslambda(z) .$$

\begin{remark}
To not cumbersome with additional notations in proving our main results, we restrict ourself to the case $n=2$, the general case can be investigated in a similar way.
\end{remark}

\section{The generalized Bergman-Dirichlet space on $\B^2$}

 The $2$-dimensional complex space $\C^2$ is endowed with the inner product $\scal{z,w}=z_1\overline{w_1}+z_2\overline{w_2}$
for $z=(z_1,z_2)$ and $w=(w_1,w_2)$ in $\C^2$. Its associated norm
is given by $|z|=\sqrt{|z_1|^2+|z_2|^2}$.
Let us denote by $S^3=\partial\B^2$ be the unit sphere of $\C^2$ viewed as the boundary of the unit ball $\B^2=\{z=(z_1,z_2)\in\C^2 \;\;/\;\; |z|^2<1\}$. On $\B^2$, we consider the weighted measure
$$\mes(z)=(1-|z|^2)^\alpha \meslambda(z),$$
where $\alpha\in\R$ and $\meslambda(z)=\left(\frac{i}{2}\right)^2dz_1\wedge
d\overline{z_1}\wedge dz_2\wedge d\overline{z_2}$ is the usual Lebesgue measure.
Notice that the measure $\mu_\alpha$ is finite on $\B^2$ if and only if $\alpha>-1$, indeed we have $$\int_{\B^2}\mes=\pi^2\frac{\Gamma(\alpha+1)}{\Gamma(\alpha+3)}.$$
For given multi-index $p=(p_1,p_2)\in\Z^+\times\Z^+$, we use as usual $|p|=p_1+p_2$, $p!=p_1!p_2!$ and set
$$z^p=z_1^{p_1} z_2^{p_2}=:\varphi_p(z).$$

In order to prove Theorem \ref{Mthm1}, we begin with the following

\begin{proposition}\label{prop1}
The monomials $\varphi_p(z)=z^p$; $p\in\Z^+\times\Z^+$, belong to $\ADm$ if and
	only if $\alpha>-1$. Moreover, for $\alpha>-1$, they form an orthogonal system in  $\ADm$ with
	\begin{align}
	||\varphi_p||_{\alpha,m}^2= \pi^2\Gamma(\alpha+1)
    \left\{ \begin{array}{ll}
    	\dfrac{p!}{\Gamma(|p|+\alpha+3)}  & \quad \mbox{ if } |p|<m\\
    	\dfrac{|p|(|p|-1)\cdots(|p|-m+1)p!}{\Gamma(|p|-m+\alpha+3)}  & \quad \mbox{ if } |p|\geq m
	\end{array}
	\right. .
    \end{align}
\end{proposition}

\begin{proof}
	Orthogonality of the monomials $\varphi_p(z)=z^p$; $p\in\Z^+\times\Z^+$, with respect to the inner product $\scal{\cdot,\cdot}_{\alpha,m}$
follows from the fact that $D^q\varphi_p(z)=\dfrac{p!}{(p-q)!}z^{p-q}$ and the orthogonality of
the monomials in classical weighted Bergman space. Furthermore, for $|p|<m$, we have
	\begin{align*}
	||\varphi_p||_{\alpha,m}^2=||\varphi_p||_{\alpha}^2=\int_{\B^2}|z^p|^2\mes .
	\end{align*}
	By integrating in polar coordinates $z=r\xi$ for $r\in[0,1[$ and $\xi\in S^{3}$, we get
	\begin{align*}
	||\varphi_p||_{\alpha}^2=\int_0^1r^{2|p|+3}(1-r^2)^\alpha dr\int_{S^{3}}|\xi^p|^2d\sigma(\xi),
	\end{align*}
	where $d\sigma$ is the	area measure on $S^3$. Now, to compute the integral
	$\int_{S^{3}}|\xi^p|^2d\sigma(\xi)$, we use the coordinates $\xi_1=e^{i\theta_1}\sin(\varphi)$,
	$\xi_2=e^{i\theta_2}\cos(\varphi)$, where $\varphi\in[0,\pi/2]$ and where $\theta_1$ and $\theta_2$ can take any
	value between $0$ and $2\pi$, we find
	\begin{align*}
	\int_{S^{3}}|\xi^p|^2d\sigma(\xi)=(2\pi)^2\int_0^{\pi/2}(\sin(\varphi))^{2p_1+1}(\cos(\varphi))^{2p_2+1}d\varphi=2\pi^2\frac{p!}{(|p|+1)!}.
	\end{align*}
	Making the change of variable $t=r^2$ yields
	$$||\varphi_p||_{\alpha}^2=\pi^2\frac{p!}{\Gamma(|p|+2)}\int_0^1t^{|p|+1}(1-t)^\alpha
	dt.$$
	The involved integral is the Euler function, which converges if and only if $\alpha>-1$.
	Now, we consider the case of $|p|\geq m$. Indeed, in this case, for $q=(q_1,q_2)$, we have
	$D^q\varphi_p(z)=\dfrac{p!}{(p-q)!}z^{p-q}$
	with the convention $D^q\varphi_p(z)=0$ if $p_1<q_1$ or $p_2<q_2$. Therefore,
	\begin{align}\label{normm}
 ||\varphi_p||_{\alpha,m}^2 &=\sum_{|q|=m}\frac{m!}{q!}\left(\frac{p!}{(p-q)!}\right)^2 ||\varphi_{p-q}||_{\alpha}\nonumber
       \\&=\pi^2\frac{p!}{\Gamma(|p|-m+2)}\int_0^1t^{|p|-m+1}(1-t)^\alpha
	dt\sum_{|q|=m}\frac{m!}{q!}\left(\frac{p!}{(p-q)!}\right).
	\end{align}	
	Thus, the norm
	$||\varphi_p||_{\alpha,m}$ is finite if and only if $\alpha>-1$. In this case, \eqref{normm} reduces further to
	\begin{align*}
	||\varphi_p||_{\alpha,m}^2= \pi^2\Gamma(\alpha+1) \left\{
	\begin{array}{ll}
	\dfrac{p!}{\Gamma(|p|+\alpha+3)}      &  \quad \mbox{if } |p|<m\\
	\dfrac{|p|(|p|-1)\cdots(|p|-m+1)p!}{\Gamma(|p|-m+\alpha+3)} & \quad \mbox{ if
	} |p|\geq m
	\end{array}
	\right. , \end{align*}
	thanks to the multi-monomial formula
\begin{equation}\label{s-nomial}
\prod_{j=0}^{k-1}(z_1+z_2-j)=k!\sum_{|p|=k}\frac{\prod_{j=0}^{p_1-1}(z_1-j)
	\prod_{j=0}^{p_2-1}(z_2-j)}{p!},
\end{equation}
which will be used systematically in the sequel.
\end{proof}

The previous proposition shows that the space $\ADm$ is non trivial if and only if $\alpha>-1$.
From now on, we assume that $\alpha>-1$.

\begin{proposition}\label{pro2} A holomorphic function $f$ belongs to $\ADm$ if and only if its
	Taylor coefficients satisfy the condition
	$$\sum_{|p|\geq m}\frac{|p|(|p|-1)\cdots(|p|-m+1)p!}{\Gamma(|p|-m+\alpha+3)}|a_p|^2<+\infty.$$
    Furthermore, we have
	$$||f||_{\alpha,m}^2=\pi^2\Gamma(\alpha+1)\sum_{p\in \N^2}\gamma_{\alpha,p}|a_p|^2,$$
	where $\gamma_{\alpha,p}$ stands for
   $$\gamma_{\alpha,p}=\left\{
	\begin{array}{ll}
	\dfrac{p!}{\Gamma(|p|+\alpha+3)} & \quad \mbox{if } |p|<m\\
	\dfrac{p!|p|(|p|-1)\cdots(|p|-m+1)}{\Gamma(|p|-m+\alpha+3)} & \quad \mbox{ if
	} |p|\geq m
	\end{array}
	\right. .$$
\end{proposition}

\begin{proof}
 Notice first that according to Proposition \ref{prop1}, we have
 $$
 ||f_{1,m}||_\alpha^2 =\pi^2\Gamma(\alpha+1)\sum_{|p|<m}\frac{p!}{\Gamma(\alpha+|p|+3)}|a_p|^2.
 $$
 For $l=(l_1,l_2)\in\Z^+\times\Z^+$ such that $|l|=m$, we have
$D^l\varphi_p(z)=\frac{p!}{(p-l)!}z^{p-l}$, thus
\begin{align*}
\int_{\B^2}|D^lf_{2,m}(z)|^2\mes(z)
&=\int_{\B^2}\left(\sum_{|p|\geq m}\frac{p!}{(p-q)!}a_{p}z^{p-q}\right)
\overline{\left(\sum_{|{p'}|\geq m}\frac{{p'}!}{({p'}-q)!}a_{{p'}}z^{{p'}-q}\right)}\mes(z)
\\&=\lim\limits_{\rho\rightarrow 1}
\int_{\B^2(\rho)}\left(\sum_{|p|,|{p'}|\geq m}\frac{p!}{(p-q)!}\frac{{p'}!}{({p'}-q)!}a_p\overline{a_{p'}}z^{p-q}\overline{z^{{p'}-q}}\right)\mes(z).
\end{align*}
From the compacticity of  $\B^2(\rho)$ and the orthogonality of $(\varphi_p)_p$ in the space $L^2(\B^2(\rho),\mes)$, we get
\begin{align*}
\int_{\B^2}|D^qf_{2,m}(z)|^2\mes(z)=\lim\limits_{\rho\rightarrow 1}
\sum_{p}\left( \frac{p!}{(p-q)!}\right)^2 |a_p|^2\int_{\B^2(\rho)}|z^{p-q}|^2\mes(z).
\end{align*}
By applying discrete monotone convergence  theorem, we obtain
$$
\int_{\B^2}|D^qf(z)|^2\mes(z)=\pi^2\Gamma(\alpha+1)\sum_{|p|\geq
	m}\left( \frac{p!}{(p-q)!}\right)^2\frac{(p-q)!}{\Gamma(\alpha+|p-q|+3)}|a_p|^2.
$$
Finally, in view of \eqref{s-nomial}, we get
$$
\sum_{|q|=m}\int_{\B^2}|D^qf_{2,m}(z)|^2\mes(z)=\pi^2\Gamma(\alpha+1)\sum_{|p|\geq
	m}\frac{(|p|(|p|-1)\cdots(|p|-q+1))p!}{\Gamma(\alpha+|p|-m+3)}|a_p|^2.
$$
This completes the proof.
	\end{proof}


In order to establish the second result in this section we need to the following

\begin{lemma} \label{lem:ev}
Let $f$ be a given holomorphic function $f$ on $\B^2$. Then, for every fixed $z\in\B^2$ we have
  \begin{equation}\label{evaluation}
   |f(z)| \leq \frac{1}{\pi^2}\frac{\Gamma(\alpha+3)}{\Gamma(\alpha+1)}\left(\sum_{k<m}(\alpha+3)_k\frac{|z|^k}{k!}+\frac{1
	}{(1-|z|^2)^{\alpha+3}}\right) ||f||_{\alpha,m}.
    \end{equation}
 Moreover,  for every compact set $K$ of $\B^2$ there is a constant $c_K$ such that
	\begin{equation}\label{contev}
     |f(z)|\leq c_K ||f||_{\alpha,m}; \quad z\in K.
     \end{equation}
\end{lemma}

\begin{proof} By Cauchy-Schwarz inequality, we get
	$$|f(z)|\leq\frac{||f||_{\alpha,m}}{\pi^2\Gamma(\alpha+1)}\left(\sum_{|p|<m}\frac{\Gamma(|p|+\alpha+3)}{p!}|z^p|^2+\sum_{|p|\geq
     m}\frac{\Gamma(|p|-m+\alpha+3)}{p!|p|(|p|-1)\cdots(|p|-m+1)}|z^p|^2\right).$$
	For $|p|\geq m$, we get $|p|(|p|-1)\cdots(|p|-m+1)\geq1$ and therefore we get
	\begin{align*}\sum_{|p|\geq m}\frac{\Gamma(|p|-m+\alpha+3)}{p!|p|(|p|-1)\cdots(|p|-m+1)}|z^p|^2
	&\leq\sum_{|p|\geq m}
	\frac{\Gamma(|p|-m+\alpha+3)}{p!}|z^p|^2
	\\&\leq \sum_{k=0}^\infty\frac{\Gamma(k+\alpha+3)}{k!}\sum_{|p|=k}\frac{|p|!}{p!}|z^p|^2
	\\&\leq \sum_{k=0}^\infty\frac{\Gamma(k+\alpha+3)}{k!}(|z|^2)^k
	\\&\leq \Gamma(\alpha+3)\frac{1}{(1-|z|^2)^{\alpha+3}}.
	\end{align*}
    Whence
	$$
    |f(z)| \leq \frac{1}{\pi^2}\frac{\Gamma(\alpha+3)}{\Gamma(\alpha+1)}\left(\sum_{k<m}(\alpha+3)_k\frac{|z|^k}{k!}+\frac{1
	}{(1-|z|^2)^{\alpha+3}}\right) ||f||_{\alpha,m}.
    $$
    Since, the function
	$z\rightarrow\sum\limits_{k<m}(\alpha+3)_k|z|^k+\frac{1 }{(1-|z|^2)^{\alpha+3}}$ is bounded  on $\B^2$ for being continuous, it follows that
    for every compact set $K$ of $\B^2$ there exists a constant $c_K$ such that
	$ |f(z)|\leq c_K ||f||_{\alpha,m}$ for every $z\in K$.
\end{proof}

We assert

\begin{proposition}
The space $\ADm$ is a Hilbert space and the monomials $\varphi_p$; $p\in\Z^+\times\Z^+$, form an orthogonal basis of it.
\end{proposition}

\begin{proof} In view of \eqref{contev}, it follows that the space $\ADm$ is a Hilbert space. Indeed,
	any Cauchy sequence $(f_p)_p$ in $\ADm$ is uniformly Cauchy sequence on every compact subset of $\B^2$. Thence, by Weierstrass' theorem,
     $(f_p)_p$  converges uniformly to a holomorphic function $f$ on $\B^2$ as well as $(D^qf_p)$ to $D^qf$.
     On the other hand, $(D^qf_p)$ is also a Cauchy sequence in the Hilbert space $L^2\left(\B^2,\mes\right)$. Thus, there exists a subsequence
	$(D^lf_{p_{p'}} )_q$ of $(D^qf_p)_p$ converging to $g\in L^2\left(\B^2,\mes(z)\right)$ pointwise almost everywhere.It
	follows that $D^qf=g\in L^2\left(\B^2,\mes\right)$ and therefore $f\in\ADm$ and $(f_p)_p$ converges to $f$ in
	$\ADm$. This proves that $\ADm$ is a Hilbert space for the norm $||.||_{\alpha,m}$. To conclude, we need only to prove that the monomials $z^p$ form a basis of $\ADm$. For this end, let $f(z)=\sum_pa_pz^p$ be a function belonging to $\ADm$ and observe that for every given integer $k$, the function
	$f_k=\sum_{|p|<k}a_pz^p$ belongs to $span\left\{\varphi_p\;;\;p\in\Z^+\times\Z^+\right\}$, the linear span of $(\varphi_p)_p$.
    Thus by Proposition \ref{pro2}, we get
	$$||f-f_k||_{\alpha,m}=\pi^2\Gamma(\alpha+1)\sum_{|p|\geq k}\gamma_{\alpha,p}|a_p|^2,$$
	for $k$ large enough. Since the involved sum is the rest of a convergent series, the sequence
	$(f_k)$ converges to $f$ with respect the norm $||.||_{\alpha,m}$. This proves that
	$$\ADm=\overline{span\left\{\varphi_p\;;\;p\in\Z^+\times\Z^+\right\}}^{||.||_{\alpha,m}}.$$ 
\end{proof}

Lemma \ref{lem:ev} shows that the evaluation map $f\longrightarrow f(z)$ is continuous, and therefore $\ADm$ is a reproducing kernel function, ccording to Riesz representation theorem. More explicitly, we have

\begin{proposition}
The reproducing kernel function of $\ADm$  is given explicitly in terms of ${_3F_2}$-sum by following closed form
\begin{align}
		K_{\ADmn}(z,w) =\frac{\Gamma(\alpha+3)}{\pi^2\Gamma(\alpha+1)}
		 & \sum\limits_{k<m}  (\alpha+3)_k\frac{\scal{z,w}^k}{k!}
		\label{RK-BDCd}
		\\ &  + \frac{\Gamma(\alpha+3)}{\pi^2\Gamma(\alpha+1)} \frac{\scal{z,w}^{m}}{(m!)^{2}} \,
		{_3F_2}\left(\begin{array}{c} 1, 1, \alpha+3\\ m+1,m+1 \end{array}\bigg |  \scal{z,w} \right) .\nonumber
		\end{align}
\end{proposition}

\begin{proof}
Since the monomials $\varphi_p$ constitute an orthogonal basis of $\ADm$, its reproducing kernel function can be computed by the formula
	\begin{align*}
K_{\ADm}(z,w) =\sum_{p\in (\Z^+)^2}\frac{\varphi_p(z)\overline{\varphi_p(w)}}{||\varphi_p||_{\alpha,m}^2}.
\end{align*}
More explicitly, by writing  $\sum_{|p|<m}= \sum_{k<m} \sum_{|p|=k} $ and $ \sum_{|p|\geq m}^\infty= \sum_{k= m}^\infty \sum_{|p|=k}$ and
   using the  multi-monomial formula $$\sum_{|p|=k}\frac{z^p\bar{w}^p}{p!} = \frac{(z_1\bar{w_1}+z_2\bar{w_2})^k}{k!} =  \frac{\scal{z,w}^k}{k!},$$
   we get
	\begin{align*}
K_{\ADm}(z,w) &= \frac{1}{\pi^2\Gamma(\alpha+1)} \left\{\sum_{|p|<m}\Gamma(|p|+\alpha+3) \frac{z^p\bar{w}^p}{p!}
	+ \sum_{|p|\geq m}^\infty 	\dfrac{\Gamma(|p|-m+\alpha+3)}{|p|(|p|-1)\cdots(|p|-m+1)} \frac{z^p\bar{w}^p}{p!}\right\}
\\ &= \frac{1}{\pi^2\Gamma(\alpha+1)} \left\{\sum_{k<m} \Gamma(k+\alpha+3)  \sum_{|p|=k}\frac{z^p\bar{w}^p}{p!}
	+ \sum_{k= m}^\infty \dfrac{\Gamma(k-m+\alpha+3)}{k(k-1)\cdots(k-m+1)}  \sum_{|p|=k} 	\frac{z^p\bar{w}^p}{p!}\right\}
\\& = \frac{1}{\pi^2\Gamma(\alpha+1)} \left\{\sum_{k<m} \Gamma(k+\alpha+3)  \frac{\scal{z,w}^k}{k!}
	+ \sum_{k= m}^\infty \dfrac{\Gamma(k-m+\alpha+3)}{k(k-1)\cdots(k-m+1)}  \frac{\scal{z,w}^k}{k!}\right\}.
\end{align*}
Therefore, making use of the Pochhammer symbol, it follows
	\begin{align*}
K_{\ADm}(z,w) &= \frac{\Gamma(\alpha+3)}{\pi^2\Gamma(\alpha+1)} \left\{\sum_{k<m} (\alpha+3)_k  \frac{\scal{z,w}^k}{k!}
	+ \sum_{k= m}^\infty \dfrac{(\alpha+3)_{k-m}}{k(k-1)\cdots(k-m+1)}  \frac{\scal{z,w}^k}{k!}\right\}
\\ &= \frac{\Gamma(\alpha+3)}{\pi^2\Gamma(\alpha+1)} \left\{\sum_{k<m} (\alpha+3)_k  \frac{\scal{z,w}^k}{k!}
	+  \scal{z,w}^{m}  \sum_{k' = 0}^\infty \dfrac{(\alpha+3)_{k'} (k'!)^2}{((k'+m)!)^2 }  \frac{\scal{z,w}^{k'}}{k'!}\right\}
\\ &= \frac{\Gamma(\alpha+3)}{\pi^2\Gamma(\alpha+1)} \left\{\sum_{k<m} (\alpha+3)_k  \frac{\scal{z,w}^k}{k!}
	+ \frac{\scal{z,w}^{m}}{(m!)^2} \sum_{k' = 0}^\infty \dfrac{(\alpha+3)_{k'} ((1)_{k'})^2}{((m+1)_{k'})^2 }  \frac{\scal{z,w}^{k'}}{k'!}\right\}
\\ &= \frac{\Gamma(\alpha+3)}{\pi^2\Gamma(\alpha+1)} \left\{\sum_{k<m} (\alpha+3)_k  \frac{\scal{z,w}^k}{k!}
	+ \frac{\scal{z,w}^{m}}{(m!)^2}  {_3F_2}\left(\begin{array}{c} 1, 1, \alpha+3\\ m+1,m+1 \end{array}\bigg |  \scal{z,w} \right) \right\}.
\end{align*}
This completes the proof.
\end{proof}

We conclude this section by noting that the proof of Theorem \ref{Mthm1} is contained in the above propositions.

\section{The generalized Bargmann-Fock spaces on $\C^2$}

In this section we prove Theorem \ref{Mthm2}. Namely, we show that the generalized Bargmann-Fock spaces are reproducing kernel
Hilbert spaces, and we explicit their reproducing kernel function. To this end, we proceed in a similar way as in the previous section. Thus, the proof of Theorem \ref{Mthm2} is contained in the following propositions, that we claim without proofs.

\begin{proposition} The monomials $\varphi_p(z):= z^p$ belongs to $\BDm$ and their square norms are given by
$$ ||\varphi_p||_{\nu,m} = 
\left(\frac{\pi}{\nu}\right)^2 \left(\frac{p!}{\nu^{|p|}}\right)
\left\{ \begin{array}{ll} 
1   & \quad  \mbox{ for } |p|<m \\
\frac{|p|(|p|-1)\cdots(|p|-m+1)}{\nu^{m}} & \quad \mbox{for
		} |p|\geq m 
 \end{array}
\right. .$$
\end{proposition}


\begin{proposition}
	The space $\BDm$ is a Hilbert space and the monomials $z^p$;
	$p\in\Z^+\times\Z^+$ constitute an orthogonal basis of
	$\BDm$. 
	\end{proposition}

\begin{proposition}
	 The space $\BDm$ is a reproducing kernel Hilbert space. Its reproducing kernel function is given by $$K_{\BDm}(z,w)=\frac{\nu^2}{\pi^2}\left(\sum_{l<m}\nu^l\frac{\scal{z,w}}{l!}+\frac{\scal{z,w}^m}{(m!)^2}
	{_2F_2} \left(\begin{tabular}{c|c}
	1\;,\;1 &  \\
	&$\nu \scal{z,w}$ \\
	m+1\;,\;m+1& \\
	\end{tabular}\right)\right)$$
	\end{proposition}

\section{Asymptotic: Proof of Theorem \ref{Mthm3}} 
The proof of Theorem \ref{Mthm3} lies essentially on following lemma giving the asymptotic of the  ${_3F_2}$-hypergeometric function. Namely,  we claim

\begin{lemma}\label{lemmc} Let $z\in\mathbb{C}$ and $a$, $b$, $c$ are complex
	number. Then we have $$\lim_{x\longrightarrow+\infty} {_3F_2}
	\left(\begin{tabular}{c|c}
	$b$\; , \;$c$\;,\; $x+a$&  \\
	& $\frac{z}{x}$ \\
	$d$ \;,\; $e$& \\
	\end{tabular}\right)={_2F_2} \left(\begin{tabular}{c|c}
	\;$b$\; , \;$c$\;&  \\
	& $z$ \\
	$d$ \;,\; $e$& \\
	\end{tabular}\right).$$
	\end{lemma}
	
The results in the third section on the unit ball can be generalized easily to the ball $\B^2_R$ centered at $0$ and of radius $R$.
Thus, the generalized Bergman-Dirichlet space is given by
$$
\ADmr=
\left\{f\in\mathcal{H}ol(\B^2_R)\;\;/\;\;||f||_{\alpha,m,R}^2=||f_{1,m}||_{\alpha,R}^2+\sum_{|l|=m}\frac{m!}{l!}||D^l
f||_{\alpha,R}^2<\infty\right\},$$ 
supplied with the norm
$$
||f||_{\alpha,R}^2=\int_{\B_R^2}|f(z)|^2\muar(z),
$$
 where
$$
\muar(z)=\left(1-\left|\frac{z}{R}\right|^2\right)^\alpha
d\lambda(z)
$$ is the density measure. We make the change of
variable $w=\frac{z}{R}$, $z\in\B^2_R$, we get
$$
\int_{\B^2_R}|z^p|\muar(z)=R^{2(|p|+2)}\int_{\B^2}|w^p|\mes(w)
=\pi^2R^{4}\frac{R^{2|p|}p!\Gamma(\alpha+1)}{\Gamma(|p|+\alpha+3)}.
$$
It follows that the norm of $\varphi_p$ in
$\ADmr$ is given by
$$
||\varphi_p||_{\alpha,R,m}=
\pi^2R^{4}\Gamma(\alpha+1)\frac{R^{2|p|}p!}{\Gamma(|p|+\alpha+3)}
$$
for $|p|<m$ and
$$
||\varphi_m||_{\alpha,R,m}^2
=\pi^2R^{4}\Gamma(\alpha+1)\frac{R^{2(|p|-m)}|p|(|p|-1)
\cdots(|p|-m+1)p!}{\Gamma(|p|-m+\alpha+3)}
$$
for $|p|\geq m$.
Moreover, the reproducing kernel function of $\ADmr$ reads
$$
K_{\ADmr}(z,w)=\frac{\Gamma(\alpha+3)}{\pi^2R^{4}\Gamma(\alpha+1)}
\sum_{k<m}\frac{(\alpha+3)_k}{R^{2k}}\frac{\scal{z,w}^k}{k!}
+\frac{\scal{z,w}^m}{(m!)^2}
{_3F_2} \left(\begin{tabular}{c|c}
$\alpha+3$\;,\;1\;,\;1 &  \\
&$ \frac{\scal{z,w}}{R^2}$ \\
m+1\;,\;m+1& \\
\end{tabular}\right).$$
Now, for the specific choose of $\alpha=\nu R^2$, we see that the density
$(1-|\frac{z}{R}|^2)^{\alpha}$ leads to the Gaussian density
$e^{-\nu |z|^2}$. Furthermore, using the Binet formula \cite[page 47]{erdedyi},
$$\frac{\Gamma(x+a)}{\Gamma(x-b)}=x^{a-b}(1+O(\frac{1}{x}))$$ 
for $x=\nu R^2$, $a=n+1$ and $b=1$, we get
$$\lim_{R\rightarrow+\infty}\frac{1}{\pi^2R^{4}}\frac{\Gamma(\nu R^2+3)}{\Gamma(\nu R^2+1)}=\left(\frac{\nu}{\pi}\right)^2.$$ 
Finally, by means of Lemma \ref{lemmc}, we deduce that $K_{\ADmr}(z,w)$ converges pointwise to $K_{\ADm}(z,w)$ as
$R\longrightarrow +\infty$.

\quad

\noindent{\bf Acknowledgement:}
The assistance of the members of the seminars "Partial differential equations and spectral geometry" is gratefully acknowledged.
A. Ghanmi and A. Intissar are partially supported by the Hassan II Academy of Sciences and Technology. M. Ziyat is partially supported by the    CNRS grant 56UM5R2015, Morocco.

\end{document}